\documentclass[12pt]{article}

\usepackage{amsfonts}
\usepackage{amssymb,amsmath}
\usepackage{amsthm}

\usepackage{todonotes}

\usepackage[upright]{fourier}\usepackage{baskervald}
\usepackage{dsfont}

\usepackage{graphicx}
\usepackage{float}
\usepackage{wrapfig}
\usepackage{subfigure}

\usepackage{url}

\usepackage{hyperref}
\usepackage{geometry}

\usepackage{verbatim}

\usepackage[english]{babel}
\usepackage[latin1]{inputenc}

\usepackage{calligra}
\usepackage[T1]{fontenc}

\usepackage{bbm}

\usepackage{mathrsfs}

\usepackage{color}
\usepackage[color,all]{xy}

\theoremstyle{definition}
\newtheorem{definition}{Definition}[section]

\theoremstyle{plain}
\newtheorem{theorem}[definition]{Theorem}

\newtheorem*{conjecture}{Conjecture}
\newtheorem*{mtheorem}{Main Theorem}

\newtheorem{lemma}[definition]{Lemma}

\newtheorem{proposition}[definition]{Proposition}

\theoremstyle{remark}
\newtheorem{example}[definition]{Example}
\newtheorem{remark}[definition]{Remark}

\newcommand{\OO}{\mathscr{O}}

\newcommand{\PP}{\mathbb{P}}
\newcommand{\mscr}[1]{\mathscr{#1}}

\newcommand{\Lscr}{L} 

\newcommand{\codim}{\text{\upshape{codim}}}

\newcommand{\Sym}{\text{\upshape{Sym}}}

\newcommand{\Pic}{\text{\upshape{Pic}}}

\newcommand{\cE}{\mathscr{E}}

\newcommand{\PPE}{{\PP(\cE)}}

\newcommand{\cHom}{\mathscr{H}\!om}

\makeatletter
\newcommand{\subjclass}[2][2010]{%
  \let\@oldtitle\@title%
  \gdef\@title{\@oldtitle\footnotetext{#1 \emph{Mathematics subject classification:} #2}}%
}
\newcommand{\keywords}[1]{%
  \let\@@oldtitle\@title%
  \gdef\@title{\@@oldtitle\footnotetext{\emph{Key words and phrases:} #1.}}%
}
\makeatother

\AtEndDocument{\bigskip{
  \textsc{Universit\"at des Saarlandes, Campus E2 4, D-66123 Saarbr\"ucken, Germany} \par  
  \textit{E-mail address}, C.~Bopp: \texttt{bopp@math.uni-sb.de} \par
  \textit{E-mail address}, M.~Hoff: \texttt{hahn@math.uni-sb.de} \par
}}

\begin{document}

\title{Resolutions of General Canonical Curves on Rational Normal Scrolls}
\author{Christian Bopp and Michael Hoff}
\date{}

\keywords{syzygy modules, relative canonical resolution, balancedness}
\subjclass{13D02, 14Q05, 14H51}

\maketitle

\begin{abstract}
Let $C\subset \PP^{g-1}$ be a general curve of genus $g$ and let $k$ be a positive integer 
such that the Brill-Noether number $\rho(g,k,1)\geq 0$ and $g > k+1$. 
The aim of this short note is to study the relative canonical resolution of $C$ 
on a rational normal scroll swept out by a $g^1_k=|\Lscr|$ with $\Lscr\in W^1_k(C)$ general. 
We show that the bundle of quadrics appearing in the relative canonical resolution is unbalanced 
if and only if
$\rho>0$ and 
$(k-\rho-\frac{7}{2})^2-2k+\frac{23}{4}>0$.
\end{abstract}

\section{Introduction}
Let $C \subset \PP^{g-1}$ be a canonical curve of genus $g$ that admits a complete base point free $g^1_k$, 
then the $g^1_k$ sweeps out a rational normal scroll $X$ of dimension $d=k-1$ and degree $f=g-k+1$. 
One can resolve the curve $C\subset \PPE$, 
where $\PPE$ is the $\PP^{d-1}$-bundle associated to the scroll $X$. 
Schreyer showed in \cite{Sch} that this so-called \emph{relative canonical resolution} is of the form 
\vspace{-2mm}
$$\vspace{-2mm}
0 \to \pi^*N_{k-2}(-k) \to \pi^*N_{k-3}(-k+2) \to \dots \to \pi^*N_1(-2) \to \OO_\PPE \to \OO_C \to 0
$$
where $\pi: C\to \PP^1$ is the map induced by the $g^1_k$ and 
$N_i=\bigoplus_{j=1}^{\beta_i}\OO_{\PP^1}(a_j^{(i)})$. 

To determine the splitting type of these $N_i$ is an open problem. 
If $C$ is a general canonical curve with a $g^1_k$ such that the genus $g$ is large compared to $k$, 
it is conjectured that the bundles $N_i$ are balanced, 
which means that $\max|a_j^{(i)}-a_l^{(i)}|\leq 1$. 
This is known to hold for $k\leq 5$ (see e.g. \cite{DP} or \cite{B}).
Gabriel Bujokas and Anand Patel  \cite{BP} gave further evidence 
to the conjecture by showing that all $N_i$ are balanced if $g=n\cdot k +1$ for $n\geq 1$ and the bundle
$N_1$ is balanced if $g\geq (k-1)(k-3)$. 
\\
The aim of this short note is to 
provide a range in which the first syzygy bundle $N_1$, hence the relative canonical resolution, 
is unbalanced for a general pair $(C,g^1_k)$ with non-negative Brill-Noether number $\rho(g,k,1)$. 
Our main theorem is the following.

\begin{mtheorem}
 Let $C\subset \PP^{g-1}$ be a general canonical curve and let $k$ be a positive integer such that 
 $\rho:=\rho(g,k,1)\geq 0$ and $g > k+1$. 
 Let $\Lscr\in W^1_k(C)$ be a general point inducing a $g^1_k=|\Lscr|$. 
 Then the bundle $N_1$ in the relative canonical resolution of $C$ is unbalanced if and only if   
 $(k-\rho-\frac{7}{2})^2-2k+\frac{23}{4}>0$ and $\rho > 0$. 
\end{mtheorem}

After introducing the relative canonical resolution, we prove the above theorem in Section \ref{bundleOfQuadrics}. 
The strategy for the proof is to study the birational image $C'$ of $C$ under 
the residual mapping $|\omega_C\otimes \Lscr^{-1}|$. 
Quadratic generators of $C'$ correspond to special generators of $C\subset \PPE$ 
whose existence forces $N_1$ to be unbalanced in the case $\rho>0$. 
Under the generality assumptions on $C$ and $\Lscr$, 
one obtains a sharp bound for which pairs $(k,\rho)$, the curve $C'$ has quadratic generators. 
Finally in section \ref{examples}, we state a more precise conjecture about the splitting type of the bundles in the relative canonical resolution.

Our theorem and conjecture are motivated by experiments using the computer algebra software \emph{Macaulay2} (\cite{M2})
and the package \texttt{RelativeCanonicalResolution.m2} \cite{BH}.


\section{Relative Canonical Resolutions} \label{relCanRes}

In this section we briefly summarize the connections between pencils on canonical curves and 
rational normal scrolls in order to define the relative canonical resolution.
Furthermore, we give a closed formula for the degrees of the bundles $N_i$ appearing in the relative canonical resolution.
Most of this section follows Schreyer's article \cite{Sch}.
\begin{definition}
Let $e_1 \geq e_2 \geq \dots \geq e_d\geq 0$ be integers, 
$\cE=\OO_{\PP^1}(e_1)\oplus\dots\oplus\OO_{\PP^1}(e_d)$ and let 
$\pi: \PPE\to \PP^1$ be the corresponding $\PP^{d-1}$-bundle.\newline
A \emph{rational normal scroll} $X=S(e_1,\dots,e_d)$ of  type $(e_1,\dots,e_d)$ is the image of 
$$
j:\PP(\cE)\to \PP H^0(\PPE,\OO_\PPE(1))= \PP^r
$$ \vspace{-1mm}
where $r=f+d-1$ with $f=e_1+\dots+e_d\geq 2$.
\end{definition}
In \cite{H} it is shown that
the variety $X$ defined above is a non-degenerate $d$-dimensional variety of minimal degree
$\deg X=f=r-d+1=\codim X+1$.
If  $e_1,\dots,e_d>0$, then $j:\PP(\cE)\to X\subset \PP H^0(\PPE,\OO_\PPE(1))= \PP^r$ is an isomorphism.
Otherwise, it is a resolution of singularities. Since 
$R^i j_* \OO_\PPE=0$, it is convenient to consider $\PPE$ instead of $X$ for cohomological considerations.
\newline
It is furthermore known, that the Picard group $\Pic(\PPE)$ is generated 
by the ruling $R=[\pi^*\OO_{\PP^1}(1)]$ and the hyperplane class
$H=[j^*\OO_{\PP^r}(1)]$ with intersection products
$$
H^d=f, \ \ H^{d-1}\cdot R=1, \ \ R^2=0.
$$

Now let $C \subset \PP^{g-1}$ be a canonically embedded curve of genus $g$ and let further 
$$g^1_k=\{D_\lambda\}_{\lambda \in \PP^1}\subset |D|$$ be a pencil of divisors of degree $k$. 
If we denote by $\overline{D_\lambda}\subset \PP^{g-1}$ the linear span of the divisor, then
$$
X=\bigcup_{\lambda\in \PP^1}\overline{D_\lambda}\subset \PP^{g-1}
$$
 is a $(k-1)$-dimensional rational normal scroll of degree $f=g-k+1$. 
 Conversely if $X$ is a rational normal scroll of degree $f$ containing a canonical curve, 
 then the ruling on $X$ cuts out a pencil of divisors 
 $\{D_\lambda\}\subset |D|$ such that $h^0(C,\omega_C\otimes \OO_C(D)^{-1})=f.$ 

\begin{theorem}[\cite{Sch}, Corollary 4.4]\label{Sch4.4}
Let $C$ be a curve with a base point free $g^1_k$ and 
let $\PPE$ be the projective bundle associated to the scroll $X$, swept out by the $g^1_k$.
\begin{enumerate}
\item [(a)]
$C\subset \PPE$ has a resolution $F_\bullet$ of type 
$$
0 \to \pi^*N_{k-2}(-k) \to \pi^*N_{k-3}(-k+2) \to \dots \to \pi^*N_1(-2) \to \OO_\PPE \to \OO_C \to 0
$$
with $\pi^*N_i=\sum_{j=1}^{\beta_i}\OO_\PPE(a_j^{(i)}R)$ and $\beta_i=\frac{i(k-2-i)}{k-1}\binom{k}{i+1}$. 
\item [(b)] The complex $F_\bullet$ is self dual, i.e., 
$\cHom(F_\bullet, \OO_\PPE(-kH+(f-2)R))\cong F_\bullet$
\end{enumerate}
\end{theorem}

\hspace{-7mm}
According to \cite{DP}, the resolution $F_\bullet$ above is called the \emph{relative canonical resolution}.

\begin{remark}
 A generalization of Theorem \ref{Sch4.4} can be found in \cite{CE} for covers $\pi:X\to Y$ of degree $k$. 
 In \cite{CE}, the authors used the Tschirnhausen bundle $\cE_T$ defined by  
 \vspace{-1mm}
 $$
 0\to \OO_Y \to \pi_*(\OO_X) \to \cE_T^{\vee} \to 0 
 $$
 to construct relative resolutions. 
 Note that for covers of $\PP^1$, $\cE_T=\cE\otimes \OO_{\PP^1}(2)$ and therefore, 
 the degrees of the syzygy bundles $N_i$ in \cite{CE} differ slightly from the ones given in 
 Proposition \ref{bundleDegs}.
\end{remark}

\begin{definition}
We say that a bundle of the form $\sum_{j=1}^{\beta_i}\OO_\PPE(nH+a_jR)$ is \emph{balanced} if
$\max_{i,j} |a_j-a_i|\leq 1$. 
The relative canonical resolution is called balanced if 
all bundles occurring in the resolution are balanced.
\end{definition}

\begin{remark}\label{remarkBalancedScroll}
To determine the splitting type of the bundle $\cE$, one can use \cite[(2.5)]{Sch}.  
It follows that the $\PP^1$-bundle $\cE$ associated to the scroll is always balanced 
for a Petri-general curve $C$ with a $g^1_k$ if $\rho(g,k,1)\geq 0$. \\
If $C$ is a general $k$-gonal curve and the degree $k$ map to $\PP^1$ is determined by a unique $g^1_k$, 
then it follows by \cite{Bal} that $\cE$ is balanced as well.
\end{remark}

\begin{remark}
 If all $a^{(i)}_j\geq -1$, \vspace{-0.5mm}
 one can resolve the $\OO_\PPE$-modules occurring in the relative canonical resolution of $C$ 
 by Eagon-Northcott type complexes. 
 An iterated mapping cone gives a possibly non-minimal resolution of the curve $C\subset \PP^{g-1}$. 
 In \cite{Sch}, Schreyer used this method to classify all possible Betti tables of canonical curves 
 up to genus $8$. 
 An implementation of this construction can be found in the \emph{Macaulay2}-package \cite{BH}.
\end{remark}

We will give a lower bound on the integers $a_j^{(1)}$ appearing 
in the resolution $F_\bullet$. 

\begin{proposition}
\label{positiveTwists}
 Let $C$ be a general canonically embedded curve of genus $g$ and 
 let $k\geq 4$ be an integer such that $\rho(g,k,1)\geq 0$ and $g>k+1$. 
 Let further $\Lscr\in W^1_k(C)$ be a general point inducing a complete base point free $g^1_k$.
 Then with notation as in Theorem \ref{Sch4.4}, 
 all twists $a_j^{(1)}$ of the bundle $N_1$ are non-negative. 
\end{proposition}

\begin{proof} 
As usual, we denote by $\PPE$ the $\PP^1$-bundle induced by the $g^1_k$. 
We consider the relative canonical resolution of $C\subset \PPE$.
Twisting of the relative canonical resolution by $2H$ 
and pushing forward to $\PP^1$,
we get an isomorphism 
$\pi_{*}(\mscr{I}_{C/\PPE}(2H))\cong N_1=\bigoplus_{j=1}^{\beta_1}\OO_{\PP^1}(a^{(1)}_j)$. 
Then, all twists $a_j^{(1)}$ are non-negative if and only if 
$$
 h^1(\PP^1,N_1(-1))=
 h^1(\PP^1,\pi_*(\mscr{I}_{C/\PPE}(2H-R)))=
 h^1(\PPE,\mscr{I}_{C/\PPE}(2H-R))=0.
$$
We consider the long exact cohomology sequence
\begin{align*}
 0 & \to H^0(\PPE,\mscr{I}_{C/\PPE}(2H-R)) \to H^0(\PPE, \OO_{\PPE}(2H-R)) \to 
 H^0(\PPE, \OO_C(2H-R))\to \\
   & \to H^1(\PPE, \mscr{I}_{C/\PPE}(2H-R)) \to \dots
\end{align*}
obtained from the standard short exact sequence.

The vanishing of $H^1(\PPE,\mscr{I}_{C/\PPE}(2H-R))$ is equivalent to
the surjectivity of the map  
$$
 H^0(\PPE,\OO_{\PPE}(2H-R))\longrightarrow H^0(C,\OO_{C}(2H-R)).
$$
From the commutative diagram 
\vspace{-4mm}
$$
\begin{xy}
\xymatrix{
 H^0(\PPE,\OO_{\PPE}(2H-R))\ar[r]  &  H^0(C,\OO_{C}(2H-R)) \\
 H^0(\PPE,\OO_{\PPE}(H))\otimes H^0(\PPE,\OO_{\PPE}(H-R)) \ar[u] \ar[r]^{\ \ \ \ \ \ \ \cong} &
 H^0(C,\OO_{C}(H)) \otimes H^0(C,\OO_{C}(H-R)) \ar[u]^{\eta}
}
\end{xy}
$$
we see that it suffices to show the surjectivity of $\eta$.

Note that the system $|H-R|$ on $C$ is $\omega_C\otimes \Lscr^{-1}$.
The residual line bundle $\omega_C\otimes \Lscr^{-1}\in W^{g-k}_{2g-2-k}(C)$ is general 
since $\Lscr$ is general.  
Hence, the residual morphism induced by $|\omega_C\otimes \Lscr^{-1}|$ is birational for $g-k\geq 2$ 
by \cite[Section 0.b (4)]{GH}. 

We may apply \cite[Theorem 1.6]{AS} and get a surjection
$$ 
\bigoplus_{q\geq 0}\Sym_q(H^0(C,\omega_C\otimes \Lscr^{-1}))\otimes H^0(C,\omega_C)\longrightarrow 
\bigoplus_{q\geq 0} H^0(C,\omega_C\otimes (\omega_C\otimes \Lscr^{-1})^q),
$$
i.e., the $\Sym(H^0(C,\omega_C\otimes \Lscr^{-1}))$-module 
$\bigoplus_{q\in \mathbb{Z}} H^0(C,\omega_C\otimes (\omega_C\otimes \Lscr^{-1})^q)$ is generated in 
degree $0$. In particular, this implies the surjectivity of $\eta$. 
\end{proof}

\begin{remark}
 Using the projective normality of $C\subset \PPE$, 
 one can show that all twists $a_j^{(1)}$ of $N_1$ are greater or equal to $-1$. 
 \vspace{0.5mm}
 There exist several examples where $N_1$ has negative twists (see \cite{Sch}). 
 We conjecture that all $a_j^{(i)}\geq -1$ and in general $a_j^{(i)}\geq 0$.
\end{remark}

It is known that the degrees of the bundles $N_i$ can be computed recursively. 
However, we did not find a closed formula for the degrees in the literature.

\begin{proposition}\label{bundleDegs}
The degree of the bundle $N_i$ of rank $ \beta_i=\frac{k}{i+1}(k-2-i)\binom{k-2}{i-1}$ 
in the relative canonical resolution $F_\bullet$ is 
\vspace{-2mm}
$$
\deg(N_i)=\sum_{j=1}^{\beta_i}a_j^{(i)}=(g-k-1)(k-2-i)\binom{k-2}{i-1}.
$$
For $i=1,2$ one obtains $\deg(N_1)=(k-3)(g-k-1)$ and $\deg(N_2)=(k-4)(k-2)(g-k-1)$.
\end{proposition}

\begin{proof}
The degrees of the bundles $N_i$ can be computed by considering the identity
\vspace{-2mm}
\begin{align}\label{alternatingEuler}
\chi (\OO_C(\nu))=\sum_{i=0}^{k-2}(-1)^i \chi( F_i(\nu)).
\end{align} 

If $b\geq -1$, we have 
\vspace{-3mm}
\begin{align*}
h^i(\PPE,\OO_{\PPE}(aH+bR))=
\begin{cases}
 h^i(\PP^1,S_a(\cE)(b)), & \text{ for } a\geq 0 \\
 0, & \text{ for } -k<a<0 \\
 h^{k-i}(\PP^1,S_{-a-k}(\cE)(f-2-b)), & \text{ for } a \leq -k
\end{cases}
\vspace{-4mm}
\end{align*}
where $f=\deg(\cE)=g-k+1$.
As in the construction of the bundles in \cite[Proof of Step B, Theorem 2.1]{CE}, 
one obtains that the degree of $N_i$ is independent of the splitting type of the bundle.  
Hence, we assume that $a_j^{(i)}\geq -1$ and therefore, 
we can apply the above formula to all terms in $F_\bullet$. 

We compute the degree of $N_n$ by induction. The base case is straightforward.
We twist the relative canonical resolution by $\OO_{\PPE}(n+1)$ and 
compute the Euler characteristic of each term. 
By the Riemann-Roch Theorem, $\chi(\OO_C(n+1))=(2n+1)g-(2n+1)$. 
Applying the above formula yields
\begin{align*}
\chi( F_i(n+1))=
\begin{cases}
 \binom{k-1+n}{k-2}+f\binom{k-1+n}{k-1}, & \text{ for } i=0 \\
 (\deg(N_i)+\beta_i)\binom{k-2+n-i}{k-2}+ \beta_i f \binom{k-2+n-i}{k-1}, & \text{ for } n\geq i \geq 1\\
 0, & \text{ for } i\geq n+1 
\end{cases}
\end{align*}
Substituting all formulas in (\ref{alternatingEuler}), we get 
\begin{align*}
(2n+1)g-(2n+1)= 
   & \binom{k-1+n}{k-2}+f\binom{k-1+n}{k-1}  \\ 
   & +\sum_{i=1}^{n-1} (-1)^i \left((\deg(N_i)+\beta_i)\binom{k-2+n-i}{k-2} + 
                                   \beta_i f\binom{k-2+n-i}{k-1}\right)  \\
   & +(-1)^n (\deg(N_n)+\beta_n).           
\end{align*}
Using the induction step, the alternating sums simplify to 
\begin{align*}
&\sum_{i=1}^{n-1} (-1)^i \deg(N_i)\binom{k-2+n-i}{k-2}= 
 (f-2)(2n+1-nk)+(-1)^{n+1}(f-2)(k-2-n)\binom{k-2}{n-1} \\
&\sum_{i=1}^{n-1} (-1)^i \beta_i\binom{k-2+n-i}{k-2}=
 k - \binom{k-1+n}{k-2} +(-1)^{n+1}\frac{k}{n+1}(k-2-n)\binom{k-2}{n-1} \\
&\sum_{i=1}^{n-1} (-1)^i \beta_i f\binom{k-2+n-i}{k-1}=
 nkf-f\binom{k-1+n}{k-1}
\end{align*}
and we get the desired formula for $\deg(N_n)$. 
\end{proof}


\section{The Bundle of Quadrics}\label{bundleOfQuadrics}

Let $C\subset \PP^{g-1}$ be a general canonically embedded genus $g$ curve 
and let $k$ be a positive integer such that the Brill-Noether number 
$\rho:=\rho(g,k,1)$ is non-negative and $g>k+1$. 
Let $\Lscr\in W^1_k(C)$ general. 
Then, we denote by $X$ the rational normal scroll swept out by the $g^1_k=|\Lscr|$ 
and by $\PP(\mscr{E})\rightarrow X$ the projective bundle associated to $X$. 
By Remark \ref{remarkBalancedScroll}, the bundle $\mscr{E}$ on $\PP^1$ is of the form 
\vspace{-3mm}
$$\vspace{-1mm}
\mscr{E}=\bigoplus_{i=1}^{k-1-\rho} \OO_{\PP^1}(1)\oplus \bigoplus_{i=1}^{\rho} \OO_{\PP^1}.
$$
By Theorem \ref{Sch4.4}, the resolution of the ideal sheaf $\mscr{I}_{C/\PP(\mscr{E})}$ is of the form 
\vspace{-3mm}
$$ \vspace{-3mm}
 0\longleftarrow \mscr{I}_{C/\PPE} \longleftarrow 
 Q:=\sum_{j=1}^{\beta_1} \OO_{\PP(\mscr{E})}(-2H + a_j^{(1)} R) 
  \longleftarrow ...
$$
where $\beta_1=\frac{1}{2}k(k-3)$. We denote $Q$ the bundle of quadrics. 
By Proposition \ref{bundleDegs}, we know the degree of $N_1=\pi_*(Q)$ is precisely
\vspace{-3mm}
$$\vspace{-1mm}
\deg(N_1)= \sum_{j=1}^{\beta_1} a_j^{(1)} = (k-3)(g-k-1). 
$$
By Proposition \ref{positiveTwists}, all $a_i$ are non-negative. 
Since each summand of $Q$ corresponds to a non-zero global section of 
$\OO_\PPE(2H-a_j^{(1)}R)$, we get $2\cdot e_1-a_j^{(1)}\geq 0$. 
Hence $a_j^{(1)}\leq 2$ for all $j$. 
It follows that the bundle of quadrics $Q$ is of the following form
$$
Q= \OO_\PPE(-2H)^{\oplus l_0}\oplus \OO_\PPE(-2H+R)^{\oplus l_1}\oplus \OO_\PPE(-2H+2R)^{\oplus l_2}.
$$
We will describe the possible generators of $\mscr{I}_{C/\PP(\mscr{E})}$ in 
$H^0(\PPE,\OO_\PPE(2H-2R))$. 
Therefore, we consider the residual line bundle $\omega_C\otimes \Lscr^{-1}$ with 
$$
 h^0(C,\omega_C\otimes \Lscr^{-1})=f=g-k+1 
\text{ and } 
\deg(\omega_C\otimes \Lscr^{-1})=2g-k-2.
$$
By \cite[Section 0.b (4)]{GH}, $|\omega_C \otimes \Lscr^{-1}|$ induces a birational map for $g > k+1$. 

\begin{lemma}
\label{correspondenceQuadricGenerator}
 Let $C'\subset\PP^{g-k}$ be the birational image of $C$ under the residual linear system 
 $|\omega_C\otimes \Lscr^{-1}|$. 
 There is a one-to-one correspondence between quadratic generators of $C'\subset\PP^{g-k}$ and 
 quadratic generators of $C\subset \PPE$ contained in  $H^0(\PPE,\OO_\PPE(2H-2R))$.
\end{lemma}

\begin{proof}
 Since $\rho\geq 0$, the scroll $X$ is a cone over the Segre variety $\PP^1\times \PP^{g-k}$. 
 Let $p:\PPE\longrightarrow \PP^{g-k}$ be the projection on the second factor. 
 An element $q$ of $H^0(\PPE,\OO_\PPE(2H-2R))$ corresponds to a global section of 
 $H^0(\PP^1,S_2(\mscr{E})\otimes \OO_{\PP^1}(-2))$ which does not depend on the fiber over $\PP^1$. 
 Hence, the image of $V(q)$ under the projection yields a quadric containing $C'$.
 Conversely, the pullback under the projection $p$ of a quadratic generator of $C'\subset\PP^{g-k}$ 
 does not depend on the fiber and has therefore to be contained in $H^0(\PPE,\OO_\PPE(2H-2R))$. 
\end{proof}

We are now interested in a bound on $k$ and $\rho$ such that the curve $C'$ lies on a quadric. 

\begin{lemma}
\label{quadraticGenerators}
 For a general curve $C$ and a general line bundle $\Lscr\in W^1_k(C)$, 
 the curve $C'\subset\PP^{g-k}$ lies on a quadric if and only if 
 the pair $(k,\rho)$ satisfies the inequality 
 $$
 (k-\rho-\frac{7}{2})^2-2k+\frac{23}{4} >0.
 $$
\end{lemma}

\begin{proof}
By \cite{W}, the map 
$$
H^0(\PP^{g-k}, \OO_{\PP^{g-k}}(2))\rightarrow H^0(C',\OO_{C'}(2))
$$ 
has maximal rank for a general curve $C$ and a general line bundle $\omega_C\otimes \Lscr^{-1}$.
Using the long exact cohomology sequence to the short exact sequence 
$$
 0\rightarrow \mscr{I}_{C'}(2)\rightarrow \OO_{\PP^{g-k}}(2) \rightarrow \OO_{C'}(2)\rightarrow 0,
$$
we see that $C'$ lies on a quadric if and only if 
$$
h^0(\PP^{g-k},\OO_{\PP^{g-k}}(2)) - h^0(C',\OO_{C'}(2)) > 0.
$$
We compute the Hilbert polynomial of $C'$: 
$
h_{C'}(n)=(2g-k-2)n+1-g 
$ 
and get $h_{C'}(2)=3g-2k-3$. 
The dimension of the space of quadrics in $\PP^{g-k}$ is $\binom{g-k+2}{2}$. 
Hence,  
\begin{align}
 \label{inequality}
 h^0(\PP^{g-k},\OO_{\PP^{g-k}} (2)) - h^0(C',\OO_{C'}(2)) = \binom{g-k+2}{2}-3g+2k+3>0. 
\end{align}
Expressing $g$ in terms of $k$ and $\rho$, the inequality (\ref{inequality}) is equivalent to 
$$
 (k-\rho-\frac{7}{2})^2-2k+\frac{23}{4} >0.
  \vspace{-4mm}
$$
\end{proof}

\begin{proof}[Proof of the Main Theorem]
As mentioned above, the bundle $Q=\pi^* N_1$ is of the form 
$
Q= \OO_\PPE(-2H)^{\oplus l_0}\oplus \OO_\PPE(-2H+R)^{\oplus l_1}\oplus \OO_\PPE(-2H+2R)^{\oplus l_2}
$
(see also Proposition \ref{positiveTwists}).
By Lemma \ref{correspondenceQuadricGenerator}, 
the bundle of quadrics is balanced if no quadratic generator of $C'\subset \PP^{g-k}$ exists. 
So, we are done for pairs $(k,\rho)$ with $(k-\rho-\frac{7}{2})^2-2k+\frac{23}{4}\leq 0$.

It remains to show that the bundle of quadrics is unbalanced in the case $\rho>0$ 
for pairs $(k,\rho)$ satisfying the inequality in Lemma \ref{quadraticGenerators}.  

Let $k$ and $\rho$ be non-negative integers satisfying the above inequality and 
let $l_2=h^0(C',\mscr{I}_{C'}(2))=(k-\rho-\frac{7}{2})^2-2k+\frac{23}{4}$ be 
the positive dimension of quadratic generators of the ideal of $C'$. 
By Lemma \ref{correspondenceQuadricGenerator}, 
the bundle $Q$ is now unbalanced if a summand of the type $\OO_{\PP(\mscr{E})}(-2H)$ exists. 
Such a summand exists if and only if the following inequality holds
\vspace{-5mm}
\begin{align}
\label{inequality2}
 l_0=\beta_1-l_2-l_1= \beta_1 -l_2 - (\sum_{i=1}^{\beta_1} a_i - 2\cdot l_2) > 0,
 \vspace{-4mm}
\end{align} 
An easy calculation shows that the inequality (\ref{inequality2}) is equivalent to 
\vspace{-2mm}
$$
 l_0=\binom{\rho+1}{2}>0.
 \vspace{-4mm}
$$\vspace{-4mm}
\end{proof}

For pairs $(k,\rho)$ in the following marked region, the bundle $Q$ is unbalanced. 

 \begin{figure}[H]
 \centering
 \includegraphics[scale=0.2]{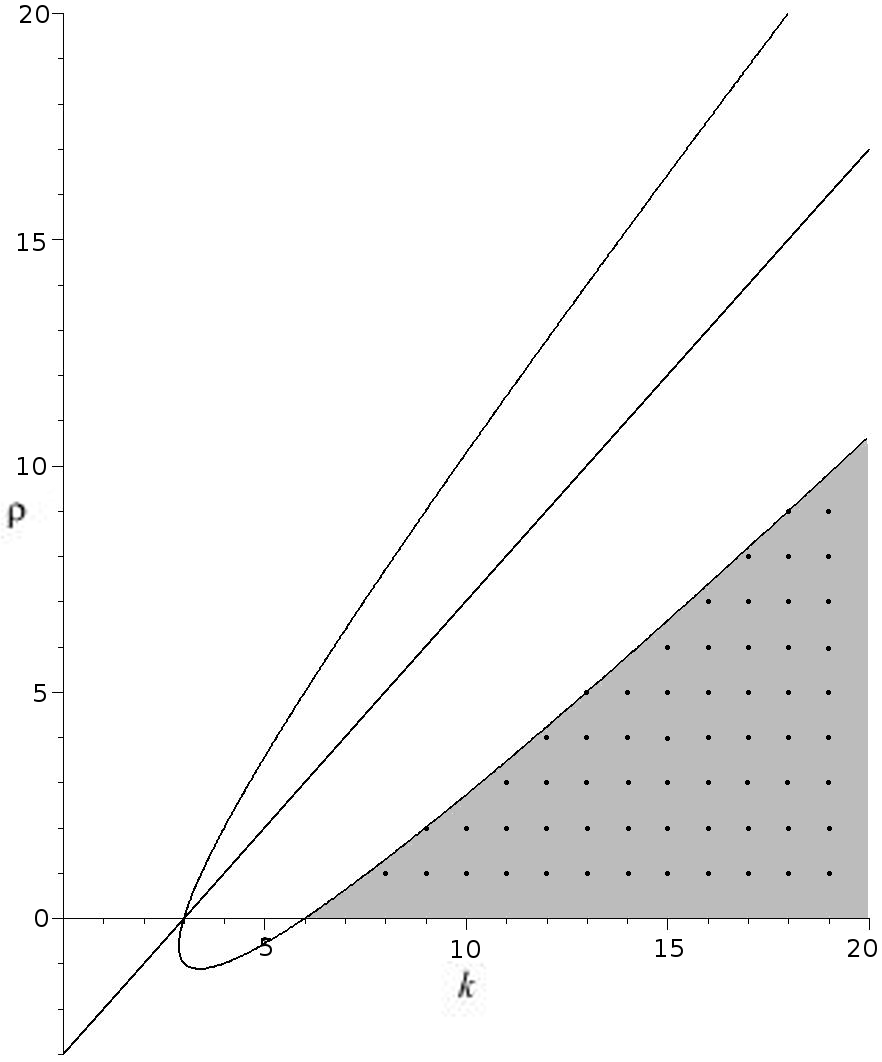}
 \caption{The conic: $(k-\rho-\frac{7}{2})^2-2k+\frac{23}{4}$=0 and 
 the line: $k-\rho-3=0\Leftrightarrow g=k+1$.}
\end{figure}
\begin{remark}
 With our presented method, the whole first linear strand of the resolution of $C'\subset \PP^{g-k}$ lifts 
 to the resolution of $C\subset \PPE$. See also Example \ref{example_g19_k11}.
\end{remark}

\section{Example and Open Problems}\label{examples}

\begin{example}
\label{example_g19_k11}
 Using \cite{BH}, we construct a nodal curve $C\subset \PP^{18}$ of genus $19$ 
 with a concrete realization of $\Lscr\in W^1_{11}(C)$. 
 The ideal of the scroll $X$ swept out by $|\Lscr|$ is given by the $2\times 2$ minors of the matrix 
 \vspace{-2mm}
 $$
 \begin{pmatrix}
  x_0 & x_2 & \dots & x_{16} \\
  x_1 & x_3 & \dots & x_{17}
 \end{pmatrix}.
 $$
 The resolution of the birational image $C'$ of $C$ under the map $|\omega_C\otimes \Lscr^{-1}|$ 
 has the following  Betti table
\begin{center}
\small
 \begin{tabular}{|c||c|c|c|c|c|c|c|c|}
 \hline
 &$0$&$1$&$2$&$3$&$4$&$5$&$6$&$7$ \\ \hline 
 $0$&$1$&-&-&-&-&-&-&- \\ \hline
 $1$&-&$13$&$9$&-&-&-&-&- \\ \hline
 $2$&-&-&$91$&$259$&$315$&$197$&$56$&$1$ \\ \hline
 $3$&-&-&-&-&-&-&-&$2$ \\ \hline
 \end{tabular}
\end{center}
 Assuming that the relative canonical resolution is as balanced as possible, 
 the first part of the relative canonical resolution is of the following form
$$
 0\leftarrow \mscr{I}_{C/\PPE}\leftarrow  
 \overset{\OO_\PPE(-2H+2R)^{\oplus 13}}{\underset{\OO_\PPE(-2H+R)^{\oplus 30}\oplus\OO_\PPE(-2H)}\oplus} \leftarrow
 \overset{\OO_\PPE(-3H+3R)^{\oplus 9}}{\underset{\OO_\PPE(-3H+2R)^{\oplus 192}\oplus\OO_\PPE(-3H+R)^{\oplus 30}}\oplus}  
 \leftarrow  \dots
$$
\end{example}
\hspace{-6mm}
Using the \emph{Macaulay2}-Package \cite{BH}, our experiments lead to conjecture the following:
\begin{conjecture}
\begin{enumerate}
\item[(a)]
 Let $C\subset \PP^{g-1}$ be a general canonical curve and let $k$ be a positive integer such that 
 $\rho:=\rho(g,k,1)\geq 0$. 
 Let $\Lscr\in W^1_k(C)$ be a general point inducing a $g^1_k=|\Lscr|$.  
 Then for bundles $N_i=\bigoplus \OO_{\PP^1}(a_j^{(i)})$, $i=2,\dots, \lceil\frac{k-3}{2} \rceil$ 
 there is the following sharp bound 
 \vspace{-2mm}
 $$
  \max_{j,l} | a_j^{(i)}-a_l^{(i)} | \leq \min\{g-k-1,i+1\} \vspace{-1mm}.
 $$ 
 In particular, if $g-k = 2$, the relative canonical resolution is balanced. 
\item[(b)] 
 For general pairs $(C,g^1_k)$ with $\rho(g,k,1)\leq 0$, the bundle $N_1$ is balanced.
\end{enumerate}
\end{conjecture}

\begin{remark}
\begin{enumerate}
 \item [(a)] In order to verify Conjecture (b), it is enough to show the existence of one curve 
 with these properties. With the help of \cite{BH}, 
 we construct a $g$-nodal curve on a normalized scroll swept out by a $g^1_k$ and 
 compute the relative canonical resolution. Then, Conjecture (b) is true for 
 \vspace{-2mm}
 $$
 (k,\rho)\in \{6,7,8,9\}\times\{-8,-7,\dots,-1,0\} \text{ where } g=2k-\rho-2.
 \vspace{-2mm}
 $$
 \item [(b)] We found several examples (e.g. $(g,k)=(17,7),(19,8),\dots$) 
 of $g$-nodal $k$-gonal curves 
  where some of the higher syzygy modules $N_i,\ i\geq 2$ are unbalanced. 
 We believe that the generic relative canonical resolution is unbalanced in these cases.
\end{enumerate}
\end{remark}

\section*{Acknowledgement}

We would like to thank Anand Patel for bringing the topic back to our attention 
and for sending a draft of his joint work with Gabriel Bujokas.  
We would further like to thank Frank-Olaf Schreyer who provided the idea of the proof of 
Proposition \ref{positiveTwists}. 
The second author was supported by the DFG-grant SPP 1489 Schr. 307/5-2. 

\bibliography{papers}{}
\bibliographystyle{alpha}

\end{document}